\documentclass{article}
\usepackage[utf8]{inputenc}

\usepackage[utf8]{inputenc}

\usepackage{comment}

\usepackage{graphicx} 
\usepackage{amsmath,amsthm,verbatim,amssymb,amsfonts,amscd, graphicx}
\usepackage[utf8]{inputenc}
\usepackage{graphics}
\usepackage{hyperref}
\usepackage[baseline]{euflag}
\usepackage{enumitem}
\usepackage{apptools}
\usepackage{titlesec}
\usepackage{appendix}
\usepackage{xcolor}
\usepackage{dsfont}
\usepackage{authblk}

\AtAppendix{\counterwithin{lemma}{section}}

\theoremstyle{plain}

\newtheorem{theorem}{Theorem}

\newtheorem{lemma}[theorem]{Lemma}
\newtheorem{proposition}[theorem]{Proposition}

\newtheorem{question}[theorem]{Question} 

\newtheorem{example}[theorem]{Example} 

\theoremstyle{definition}

\newcommand{\RL}{\mathbb{R}}

\newcommand{\dd}{\,d}
\newcommand{\1}{\mathbf 1}

\usepackage[textwidth=2cm]{todonotes}

\DeclareMathOperator{\supp}{supp}

\title{Entropic analogues of Gr{\"u}nbaum's inequality}

\date{}

\author[1]{Matthieu Fradelizi}
\author[2]{Lampros Gavalakis}
\author[3]{Martin Rapaport}

\affil[1]{Univ Gustave Eiffel
                      }
                      
\affil[2]{University of Cambridge 
}

\affil[3]{
Carnegie Mellon University
                }

\begin{document}

\maketitle
 \footnotetext[1]{M.F. is with Univ Gustave Eiffel, Univ Paris Est Creteil, 
 CNRS, LAMA UMR8050 F-77447,
 Marne-la-Vall{\'e}e, France.\\
 email: matthieu.fradelizi@univ-eiffel.fr 
                      
L.G. is with the Department of Pure Mathematics and Mathematical Statistics, University of Cambridge,  UK and was supported by the Engineering and Physical Sciences Research Council [Grant Ref:EP/Y028732/1].\\
email: lg560@cam.ac.uk  

M.R. is with the 
Department of Mathematical Sciences, Carnegie Mellon University, Pittsburgh, 15213, PA, United States. \\
email: mrapapor@andrew.cmu.edu }
 
\begin{abstract} 
The classical Gr{\"u}nbaum inequality asserts that the proportion of the volume of a convex body cut off by a halfspace containing its barycenter is at least $1/e$. From its functional counterpart, for any log-concave random variable $X$, one has $\mathbb{P}(X\ge \mathbb{E}X)\ge 1/e$, with equality if and only if $X$ is exponential. 
Motivated by Gr{\"u}nbaum's inequality for convex bodies and its functional generalizations, we prove analogous inequalities for entropy, with characterizations of the equality cases. We show that if $X$ is a log-concave random variable on $\mathbb{R}$, then
$$
        h(X)-\frac{e}{e-1}H_2(1/e) \leq h(X|X \leq \mathbb{E}X) \leq h(X),
$$
where $h$ is the differential entropy, $H_2(\cdot)$ is the binary entropy function and $X|X\leq \mathbb{E}X$ stands for the distribution of $X$ conditional on $X\leq \mathbb{E}X$. We generalize the upper bound for all R{\'e}nyi entropies and the lower bound for {\em min}--entropy. Our inequalities are sharp and we characterize all equality cases. 

We discuss potential generalizations in high dimensions and give counterexamples in some directions. 

As an intermediate step for the proof of the lower bound, we establish a new inequality that we prove using a technique known as \textit{degrees of freedom}, combined with a standard KKT-type optimization lemma. Along the way, we characterize the equality case in a known comparison inequality between differential and {\em min}--entropy, which may be of independent interest.  
\end{abstract}

\section{ Introduction and main results}

Let $X \in \RL^n$ be a log-concave random vector. Let $H$ be an $(n-1)$-dimensional hyperplane through the mean of $X$, and denote by $H^-$ and $H^+$ the two half-spaces in $\mathbb{R}^n$ whose boundary is $H$. The functional version of Gr{\"u}nbaum's inequality \cite{grunbaum} states that
\begin{equation} \label{grunbaumeq}
    \frac{1}{e} \leq \mathbb{P}(X\in H^+),\hspace{0.2cm} \mathbb{P}(X\in H^-) \leq 1-\frac{1}{e}.
\end{equation}

In view of the analogy between entropy and volume \cite{costacover, dembocoverthomas, FMMZ}, the growing role of entropy methods in convex geometry \cite{ bobkovmadiman, bourgains,klartagstrong}, as well as the recently established entropic analogues of Bergstr{\"o}m's and Bonnesen's inequalities \cite{bergstromisit}, it is natural to ask whether Gr{\"u}nbaum's inequality admits an entropic counterpart. More precisely, we address the question of how the entropy of a log-concave distribution behaves when restricting to a half-space.

Let $X$ be a random vector with log-concave density $f$. For a hyperplane $H$, we write
\[
    X^+:=(X\mid X\in H^+)
    \quad \text{ and } \quad
    X^-:=(X\mid X\in H^-).
\]
Thus, for instance, $X^+$ has density
\[
    f^+(x)
    :=
    \frac{f(x)\mathbf 1_{H^+}(x)}
         {\mathbb P(X\in H^+)}.
\]
The differential entropy of $X$ is 
$$
h(X) = -\int_{\mathbb R^n}f(x) \log f(x) dx.
$$
We will abuse the notation and write $h(f) = h(X)$ when it is convenient. Denoting $p^+ := \mathbb{P}(X\in H^+)$ and $p^- := 1-p^+,$ a direct computation gives
\begin{equation} \label{elementaryid}
    h(X) = p^+h(X^+) +p^-h(X^-) + H_2(p^+), 
\end{equation}
where $H_2(p) := -p\log{p} - (1-p)\log{(1-p)}, p \in [0,1]$.
From an information theoretic standpoint, this can also be seen from the definition of mutual information (see \eqref{mutualinfodef} in the Appendix), since
\begin{align} \label{plusminuseq}
I(X;\mathbf{1}_{\{X\in H^+\}}) &= h(X)-h(X|\mathbf{1}_{\{X\in H^+\}}) = h(X) - p^+h(X^+) -p^-h(X^-) \\ \label{IH}
&= H(\mathbf{1}_{\{X\in H^+\}}) = H_2(p^+).
\end{align}

We aim to understand the relationship between $h(X^+), h(X^-)$ and $h(X)$. In the one-dimensional setting, we see from \eqref{elementaryid}, that if $H^+=\mathbb{R}_+$ and $f$ is symmetric around $0$, then 
$$
h(X^+) = h(X^-) = h(X)-\log{2}.
$$
On the other hand, for the centered exponential with density $f(x)=e^{-x-1}\mathbf{1}_{[-1,\infty)}(x),$ the memoryless property gives
\begin{equation} \label{exponentialvalue}
h(X^+) = h(X) = 1 = h(X^-) + \frac{e}{e-1}H_2\Bigl(\frac{1}{e}\Bigr),
\end{equation}
since this distribution also extremizes \eqref{grunbaumeq}.

Motivated by these examples, while thinking of the exponential distribution as an extreme point of the set of log-concave measures under moment constraints~(see, e.g., \cite{MNR}), it is natural to conjecture that in dimension $1$
\begin{equation} \label{condstrengthen}
    \max{\{h(X|X\in H^+),h(X|X\in H^-)\}} \leq h(X). 
\end{equation}

Our main contributions are twofold. First, we prove \eqref{condstrengthen} in dimension $1$, as well as the more general statement that
R\'enyi entropy of any order (defined just below) is non-increasing under one-sided truncation of a
one-dimensional log-concave density. Furthermore, we characterize all equality
cases. Second, under a centering assumption, we establish sharp reverse
inequalities for differential entropy and min-entropy, again with
complete equality characterizations.

We recall here that the R\'enyi entropy of order
$\alpha\in[0,\infty]$ of a random variable with density $f$ is
\begin{equation} \label{renyidef}
    h_\alpha(f)
    =
    \frac{1}{1-\alpha}
    \log\left(\int_{\mathbb R}f(x)^\alpha\,dx\right),
\end{equation}
with the usual limiting definitions at $\alpha=1$ and
$\alpha=\infty$. In particular,
\[
    h_1(f)=h(f)
    \quad \text{ and } \quad
    h_\infty(f)=-\log\|f\|_\infty,
\]
where $h_\infty$ is commonly referred to as min-entropy.

Our first main result is



\begin{theorem}
\label{upperThIntro}
Let $X \in \mathbb R$ have a log-concave density $f$. For $m \in \mathbb R$, let
\[
f^{-}_m(x):=\frac{f(x)\mathbf 1_{(-\infty,m]}(x)}{\int_{-\infty}^{m} f(y)\,dy}, \quad x \in \mathbb R,
\]
be the density of $X|X\le m$.
Then, for every $\alpha\in[0,\infty]$, and $m \in \mathbb R$ we have
\begin{equation} \label{alphaineqIntro}
h_\alpha\!\big(f^{-}_m\big)\le h_\alpha(f).
\end{equation}

Moreover, for any $\alpha \in (0,\infty]$, there is equality in \eqref{alphaineqIntro}, if and only if either the support of $X$ is contained in $(-\infty, m]$, or $X$ has an exponential distribution on $(-\infty,M]$ for some $M \geq m$. 
\end{theorem}

Repeatedly applying the conclusion of Theorem \ref{upperThIntro} shows that, for every $\alpha\in [0,+\infty]$, 
$$
    t \mapsto h_{\alpha}(X|X \leq t ) \text{ is non-decreasing in }t,
$$
or equivalently 
\begin{equation} \label{residualfunc}
    t \mapsto h_{\alpha}(X|X \geq t ) \text{ is non-increasing in }t.
\end{equation} 

The differential entropy case is the key step in the proof of
Theorem~\ref{upperThIntro}. The general R\'enyi inequality is deduced
from it by applying the differential entropy result to suitable normalized powers of the density. The equality analysis relies on a
rigidity result for the comparison inequality
$$
\log\|g\|_\infty \le 1+\int_{\mathbb R} g\log g,
$$
which was established in \cite[Proof of Theorem 7]{fradelizi2008increasing} and is an important ingredient in the proof of the inequality for differential entropy, Theorem \ref{entropygrunbaumTh}. We characterize all
equality cases in this inequality in
Proposition~\ref{prop:fradelizi-equality-1d}. As another consequence of
the min-entropy case of Theorem~\ref{upperThIntro} and Gr\"unbaum's
inequality, one recovers the known estimate
\cite[Theorem 4]{fradelizisections}, 
\[
    \|f\|_\infty\leq e f(0)
\]
for every centered log-concave probability density on $\mathbb R$.

In Section \ref{grunbaumineqsec} we break Theorem \ref{upperThIntro} into 
Proposition \ref{min-entroy}, 
Theorem \ref{entropygrunbaumTh},
Theorem \ref{thm:renyi-truncation}
and prove  
the equality case characterizations 
in 
Proposition 
\ref{cor:renyi-truncation-equality} using Proposition \ref{prop:fradelizi-equality-1d}. 


Unlike the upper bound in Theorem~1, a reverse inequality requires the truncation point to be centered at the mean. Combining the $\alpha =1$ case of Theorem \ref{upperThIntro} (applied to $X^+$) with the identity \eqref{elementaryid} and Gr\"unbaum's
inequality gives the non-sharp reverse estimate
\[
        h(X)\le h(X^-) + e\,H_2(1/e).
\]
Our second main result identifies the sharp constant, both for
differential entropy and for min-entropy:

\begin{theorem}
\label{secondThIntro}
Let \(X\) be a random variable on $\mathbb R$ with zero mean and log-concave density \(f\), so that
\[
        \int_{\mathbb R} xf(x)\,dx=0.
\]
Let $\alpha \in \{1,\infty\}$ and $
f^{-}(x):=\frac{f(x)\mathbf 1_{(-\infty,0]}(x)}{\int_{-\infty}^{0} f(y)\,dy},  x \in \mathbb R
$.
Then
\begin{equation} \label{grunbaumineqintro}
        h_{\alpha}(f^-)
        \ge
        h_{\alpha}(f)-C_{\alpha},
\end{equation}
where 
$$
C_1 = \frac{e}{e-1}H_2(1/e)=\frac{e}{e-1}-\log(e-1) \quad \text{ and }  \quad C_{\infty} = \log(1+\sqrt{2}).
$$
For $\alpha = 1$, there is equality in \eqref{grunbaumineqintro} if and only if $X$ follows a centered exponential distribution on $[\frac{1}{\lambda}, \infty)$ for some $\lambda < 0$.

For $\alpha = \infty,$ there is equality in \eqref{grunbaumineqintro} if and only if $f$ is constant on a finite interval $[0,a], a \geq 0$ and exponential on $(-\infty, 0]$, or equivalently, if and only if $X^+$ is uniformly distributed on an interval and $X^-$ is exponentially distributed.
\end{theorem}
The extremizers in Theorem~\ref{secondThIntro} depend on the R\'enyi entropy
order. We remark first that the extremizer in \eqref{grunbaumineqintro} for $\alpha=1$ is the mirrored exponential to the one that extremizes \eqref{alphaineqIntro} in Theorem \ref{upperThIntro}, and  for $\alpha=+\infty$, the extremizer is instead
piecewise log-affine: it is exponential on one side of the origin and
constant on the other.

In Section \ref{grunbaumSec}, we break Theorem \ref{secondThIntro} into
Theorem \ref{minentropyreverseTh} for {\em min}--entropy and
Theorem \ref{thm:sharp-reverse-grunbaum-shannon} for differential entropy. 

The proof of the differential entropy inequality relies on a new
constrained entropy minimization result for log-concave densities on
$\mathbb R_+$. More precisely, an important ingredient is
Proposition~\ref{prop:normalized-variational-inequality}. This is reminiscent of the result of Melbourne, Nayar,
and Roberto \cite{MNR}, who identified one-sided exponentials as entropy
minimizers among log-concave random variables with fixed variance. Using the
\textit{degrees of freedom} method of \cite{fradeliziguedon}, we first reduce the
problem to densities that are log-affine on at most three intervals.
We then apply the Fritz--John stationarity conditions of
Lemma~\ref{FJ} to show that any minimizer must in fact be log-affine on
a single interval, for which the desired inequality can be verified
directly.


\medskip

{\bf Higher-dimensional extensions.} One might hope that Theorems \ref{upperThIntro} and \ref{secondThIntro} extend to higher dimensions. In particular, it is tempting to conjecture that \eqref{alphaineqIntro} and \eqref{grunbaumineqintro} remain true in any dimension. There are, however, two distinct obstructions.

First, Example \ref{counterexample} at the end of Section \ref{grunbaumineqsec}, shows that, already for min-entropy in dimension $2$, \eqref{alphaineqIntro} can fail if one does not require the boundary of the truncation to pass through the mean. 

Second, even under this additional assumption, the following CLT argument shows that \eqref{alphaineqIntro} cannot hold without an additive dimensional constant if the dimension is large enough. Indeed, let $X=(X_1,\ldots,X_n)$ where $X_1,\ldots ,X_n$   are independent centered exponential random variables, i.e., each $X_i$ has density
$
f_i(x)=e^{-(x+1)}\mathbf 1_{\{x\ge -1\}}(x).
$
Then \(X\) is centered and log-concave, and
\[
h(X)=n.
\]
Let $H^-=\{x=(x_1,\dots,x_n)\in\mathbb{R}^n; \sum_{i=1}^nx_i\le0\}$ and $X^-=X|X\in H^-.$
The density of \(X^-\)  is
\[
f_{X^-}(x)
=
\frac{e^{-\sum_{i=1}^n(x_i+1)}}
{\mathbb P(S_n\le 0)}
\mathbf 1_{\{x_i\ge -1,\ \sum_{i=1}^n{x_i}\le 0\}}.
\]
Therefore, letting $S_n:=\sum_{i=1}^n X_i$, we have
\[
h(X^-)
= -\int_{\mathbb R^n}f_{X^-}(x)\log{f_{X^-}(x)} \ dx =  
n+\mathbb E[S_n\mid S_n\le 0]
+
\log\mathbb P(S_n\le 0).
\]
By the CLT, as $n \to \infty,$
$
\mathbb P(S_n\le 0)\to \frac12,
$ 
and  
\[
\mathbb E[S_n\mid S_n\le0]
=-\sqrt{\frac{2n}{\pi}}+o(\sqrt n).
\]
Thus, 
\[
h(X^-)
=
n-\sqrt{\frac{2n}{\pi}}+o(\sqrt n)=h(X)-\sqrt{\frac{2n}{\pi}}+o(\sqrt n).
\]
Introducing $H^+=\{ x=(x_1,\dots,x_n)\in\mathbb{R}^n;\sum_{i=1}^nx_i\ge0\}$, $X^+=X|X\in H^+$, $p^{-}=\mathbb P(S_n\le 0)$ and $p^{+}=\mathbb P(S_n\ge 0)$. From \eqref{elementaryid}, we get
\[
h(X^+)=\frac{1}{p^+}\left(h(X)-p^{-}h(X^-)-H_2(p^+)\right)=h(X)+\sqrt{\frac{2n}{\pi}}+o(\sqrt n).
\]
Thus $h(X^+)-h(X)=\sqrt{\frac{2n}{\pi}}+o(\sqrt n)>0$ when $n$ is sufficiently large. Thus, even under the additional assumption that the truncating hyperplane contains the mean, there cannot be a direct dimensional extension of \eqref{alphaineqIntro} for $\alpha =1$.
This leads to the following question.
\begin{question}
    What are the best constants $c_n, C_n>0$ such that for every log-concave random vector $X$ in $\mathbb R^n$ and for hyperplane $H$ containing $\mathbb E X$, one has
    $$
    h(X)-c_n\le h(X|X\in H^+) \leq h(X) + C_n ?
    $$
\end{question}
\noindent
We believe that the extremizers of the above question should be product exponential density as above. 
This leads to the following question, which we believe is of independent interest: 
\begin{question} \label{qtheta}
    Let $X = (X_1,\ldots,X_n)$ where $X_i$ are i.i.d. centered exponentials as above. Let $\theta \in \mathbb{S}^{n-1}$. Is it true that 
    $$
    \theta \mapsto h(X|\langle X,\theta \rangle \leq 0) 
    $$
    is minimized at the diagonal $\theta_1 := (1/\sqrt{n},\ldots,1/\sqrt{n})$ and maximized at the opposite diagonal $-\theta_1$?
\end{question}
\noindent
Question \ref{qtheta} is related in spirit to the treatment \cite{brazitikosexp}, as well as the results in \cite{eskenazisg,ENT} on the entropy of Gaussian mixtures. However, here we aim to extremize with respect to the conditioning event.   

The  argument above shows that any dimensional extension of Theorem \ref{secondThIntro} with constants $C_{\alpha,n}$ would have to satisfy $C_{1,n} = \Omega(\sqrt{n})$.
By letting $X$ be uniform on a centered convex body $K$, any such generalization of \eqref{grunbaumineqintro} would give a Gr{\"u}nbaum inequality of the form
$$
\mathrm{Vol(K^-)} \geq \frac{1}{e^{C_{\alpha,n}}}\mathrm{Vol}(K). 
$$
This estimate is suboptimal even in dimension $1$ (with the multiplicative constant $e^{-C_{\infty}}$) when compared to the optimal $\frac{1}{2}$ for an interval. Nevertheless, the constants in Theorem \ref{secondThIntro} are sharp, and we refer to \eqref{grunbaumineqintro} as an entropy analogue of Gr{\"u}nbaum's inequality (in dimension $1$). 
We do not know whether it is possible to recover the constant $\frac{1}{2}$ for intervals in dimension $1$ by a more clever application of the entropic estimate.


\medskip
\textbf{Other connections and applications.} Combining Gr{\"u}nbaum's inequality \eqref{grunbaumeq} with the elementary property \eqref{IH}, we see that for a log-concave $X$, 
$$
I(X;\mathbf{1}_{\{X\in H^+\}}) \geq H_2\Bigl(\frac{1}{e}\Bigr).
$$
In a similar spirit, \eqref{condstrengthen} in $d=1$  strengthens the standard conditioning-reduces-entropy inequality (see \eqref{conditionreduceentropy} in the Appendix). 

Another way to state \eqref{grunbaumeq} from a statistical point of view is that trimming away a set of extreme values of probability at most $\frac{1}{e}$ will still keep the mean in the retained region, assuming the data distribution is log-concave. From this standpoint, \eqref{residualfunc} says that, in dimension one, assuming the data distribution is log-concave, removing a half-space of values never increases uncertainty.

Let us remark that the function in \eqref{residualfunc} has been studied in the context of reliability engineering \cite{residual}. There, distributions with the property \eqref{residualfunc} are said to have decreasing uncertainty of residual life. Hence, our results imply that log-concave distributions in one dimension have decreasing uncertainty of residual life. The same function has been used as a basis for goodness-of-fit test for the exponential distribution~\cite{goodness}. See also \cite{hazard, piraeus, testinglifetime, chernoffresidual, propertiesofShannon, raocumulative} for this line of work.


\subsection*{Statement on the use of A.I.}

The argument in the proof of Proposition \ref{prop:normalized-variational-inequality} and in particular the idea to use the stationarity conditions of Lemma \ref{FJ} to conclude that every local minimizer is one--piece log-affine was given by ChatGPT 5.4 and 5.5 after several prompts. The final proof was verified and written by the authors. 

\section{R{\'e}nyi entropies do not increase under truncation} \label{grunbaumineqsec}

In this section, we prove Theorem~\ref{thm:renyi-truncation}.
We first establish the differential entropy case in Section \ref{diffentropysubsec}, from which the general
Rényi inequality follows. We prove the general inequality and then characterize the equality cases in Section \ref{renyisubsec}. We
conclude with a two-dimensional example illustrating the necessity
of the centering assumption in possible higher-dimensional extensions in Section \ref{counterexamplesec}.

\subsection{The differential entropy case} \label{diffentropysubsec}

\begin{lemma}
Let $f:\mathbb{R}\to[0,\infty)$ be a log-concave probability density. 
Then the function
\[
x\longmapsto -\frac{\int_{-\infty}^x f(y)\log f(y)\,dy}{\int_{-\infty}^x f(y)\,dy}+\log \Biggl(\int_{-\infty}^x f(y)\,dy\Biggr)
\]
is non-decreasing on $\mathbb{R}$.
\end{lemma}

\begin{proof}
Fix $x\in\mathbb{R}$ and define 
\[
F(x):=\int_{-\infty}^x f(y)\,dy,
\qquad
\Phi(x):=-\int_{-\infty}^x f(y)\log f(y)\,dy \quad \text{ and } \quad \psi(x):=\frac{\Phi(x)}{F(x)}.
\]
Consider also the probability density
\[
g_x(y):=\frac{f(y)}{F(x)}\mathbf{1}_{\{y\le x\}}.
\]
Since $f$ is log-concave, so is $g_x$. By the entropy bound for log-concave densities~\cite{fradelizi2008increasing}
\begin{equation} \label{matthieusineq}
\log\big(\|g_x\|_\infty\big)\le 1+\int g_x(y)\log g_x(y)\,dy.
\end{equation}
Noting that $\|g_x\|_\infty\ge g_x(x)=f(x)/F(x)$ and that
\[
\int g_x(y)\log g_x(y)\,dy
=
\frac{1}{F(x)}\int_{-\infty}^x f(y)\log f(y)\,dy-\log F(x),
\]
we obtain
\[
\log f(x)\le 1+\frac{1}{F(x)}\int_{-\infty}^x f(y)\log f(y)\,dy.
\]
A direct computation shows that
\[
\psi'(x)
=
\frac{f(x)}{F(x)}
\left(
-\log f(x)
+
\frac{1}{F(x)}\int_{-\infty}^x f(y)\log f(y)\,dy
\right),
\]
and the previous estimate yields
\[
\psi'(x)\ge -\frac{f(x)}{F(x)}.
\]
Therefore,
\[
(\psi+\log F)'(x)
=
\psi'(x)+\frac{f(x)}{F(x)}
\ge 0,
\]
which ends the proof.
\end{proof}





\begin{theorem}\label{entropygrunbaumTh}
Let $f$ be a log-concave probability density and for $m \in \RL$, let 
\[
f_{m}^-(x):=\frac{f(x)\mathbf 1_{(-\infty,m]}(x)}{\int_{-\infty}^{m} f(y)\,dy}.
\]
Then
\[
h(f_{m}^-)\le h(f).
\]
\end{theorem}

\begin{proof}  For \(x\in\mathbb R\), we define as before
\[
F(x):=\int_{-\infty}^x f(y)\,dy,\quad
\Phi(x):=-\int_{-\infty}^x f(y)\log f(y)\,dy \quad \text{ and } 
\psi(x):=\frac{\Phi(x)}{F(x)}.
\]
By the previous Lemma, the function
\[
x\mapsto \psi(x)+\log F(x)
\]
is non-decreasing. Denote $p_m^{-}:=\int_{-\infty}^{m} f(x)\,dx$. By definition, 
\[
\psi(m)+\log F(m)
=
-\frac1{p_m^{-}}\int_{-\infty}^m f(x)\log f(x)\,dx+\log p_m^{-}
=
h(f_{m}^-).
\]
On the other hand, since \(F(x)\to1\) as \(x\to\infty\), we have
\[
\lim_{x\to\infty}\big(\psi(x)+\log F(x)\big)
=
-\int_{\mathbb R} f(x)\log f(x)\,dx
=
h(f).
\]
Therefore,
\[
h(f_{m}^-)\le h(f).
\]

\end{proof}


\subsection{Rényi entropies and equality cases} \label{renyisubsec}

Since $\supp(f_{m}^-)\subset \supp(f)$, it is immediate that $h_{0}(f_{m}^-)\leq h_{0}(f)$. We first treat the min-entropy case.
\begin{proposition}
\label{min-entroy}
Let $f$ be a log-concave probability density and for $m \in \RL$, let 
\[
f_{m}^-(x):=\frac{f(x)\mathbf 1_{(-\infty,m]}(x)}{\int_{-\infty}^{m} f(y)\,dy}.
\]
Then
\[
h_{\infty}(f_m^-)\le h_{\infty}(f).
\]
Equivalently,
\[
\|f_m^-\|_\infty \ge \|f\|_\infty .
\]
\end{proposition}

\begin{proof}
Without loss of generality, assume that $m=0$ and define as shorthand
$
p_{-}:=\int_{-\infty}^{0} f(x)\,dx.$
We omit the index $0$ in $f_0^-$ and we write
\[
f^{-}(x):=\frac{f(x)\mathbf 1_{(-\infty,0]}(x)}{p_{-}}.
\]
By definition,
\[
\|f^-\|_\infty
=
\frac{1}{p_{-}}\sup_{x\le 0} f(x).
\]
Thus the claim is equivalent to
\[
\sup_{x\in\mathbb{R}} f(x)
\le
\frac{\sup_{x\le 0} f(x)}{\int_{-\infty}^{0} f(x)\,dx}.
\]

Let $t_{\max}\in\mathbb{R}$ be such that $f(t_{\max})=\sup_{x\in\mathbb{R}} f(x)$. If $t_{\max}\le 0$, then $\sup_{x\le 0} f(x)=\sup_{x\in\mathbb{R}} f(x)$ and $ p_{-}\le 1$, so
\[
\frac{\sup_{x\le 0} f(x)}{p_{-}}\ge \sup_{x\in\mathbb{R}} f(x),
\]
holds trivially. Assume now that $t_{\max}>0$.  
As before, define
\[
F(x):=\int_{-\infty}^{x} f(y)\,dy.
\]
Since $f$ is log-concave,  Pr\'ekopa's theorem implies that $F$ is log-concave on $\mathbb{R}$. Therefore $\log F$ is concave, and its derivative
\[
(\log F)'(x)=\frac{F'(x)}{F(x)}=\frac{f(x)}{F(x)}
\]
is non-increasing on $\mathbb{R}$. In particular, for $0\le t_{\max}$ we have
\[
\frac{f(0)}{F(0)}\ge \frac{f(t_{\max})}{F(t_{\max})}\geq \frac{\sup_{x\in \mathbb{R}}f(x)}{\int_{-\infty}^{\infty} f(x)dx}
=
\sup_{x\in\mathbb{R}} f(x).
\]
If there is equality then $\int_{-\infty}^{\infty} f(x)dx=F(t_{\max})$ and so $f(x) = 0$ on $(t_{\max},+\infty)$. Moreover $f/F=F'/F$ is constant on $[0,t_{\max}]$, thus there exists $c\in\mathbb{R}$ such that $F(t)=F(0)e^{ct}$, for every $t\in[0,t_{\max}]$. Hence $f(t)=F'(t)=F(0)ce^{ct}$. By log-concavity of $f$, this implies that $f(t)\le F(0)ce^{ct}$, for every $t\le t_{\max}$. Hence, integrating, we deduce that $F(t)\le F(0)e^{ct}$, for every $t\le t_{\max}$. Since there is equality for $t=t_{\max}$, we conclude that $f(t)= F(0)ce^{ct}$, for every $t\le t_{\max}$.
\end{proof}

As an immediate consequence of Proposition \ref{min-entroy} and the functional
form of Gr\"unbaum's inequality, we recover the estimate
\[
\|f\|_\infty \leq e f(0)
\]
for every centered log-concave probability density $f$ on $\mathbb R$;
see \cite[Theorem~4]{fradelizisections}.
Indeed, without loss of generality, assume that a maximizer $t_{\max}$ of $f$ satisfies $t_{\max}\ge 0$.
By Proposition~\ref{min-entroy}, we have
\[
\sup_{x\in\mathbb{R}} f(x)\le \frac{f(0)}{F(0)},
\qquad
F(0)=\int_{-\infty}^{0} f(x)\,dx.
\]

The functional Grünbaum inequality (see \cite[Lemma 5.12]{lovasz}) states that for any log-concave probability density on
$\mathbb{R}^n$ and any half-space $H$ containing its centroid,
\[
\int_H f(x)\,dx \ge \frac{1}{e}.
\]
Applying this inequality in dimension $n=1$ with centroid $0$ and halfspace $H=(-\infty,0]$, we obtain
$F(0) \ge \frac{1}{e}.$ Thus,
\[
\sup_{x\in\mathbb{R}} f(x)\le \frac{f(0)}{F(0)}\le e\,f(0).
\]

\medskip

Recall the definition of R{\'e}nyi entropy in \eqref{renyidef}.
We have treated the endpoint cases ($\alpha=0,\alpha=1$ and $\alpha=\infty$). Next we state the general result: 
\begin{theorem}
\label{thm:renyi-truncation}
Let $f$ be a log-concave probability density and for $m \in \RL$, let 
\[
f_{m}^-(x):=\frac{f(x)\mathbf 1_{(-\infty,m]}(x)}{\int_{-\infty}^{m} f(y)\,dy}.
\]
Then, for every $\alpha\in[0,\infty]$, 
\[
h_\alpha\!\big(f_{m}^{-}\big)\le h_\alpha(f).
\]
\end{theorem}

\begin{proof}
The cases $\alpha=0$ and $\alpha=\infty$ were established above.
We may therefore assume that $\alpha\in(0,\infty)$ and, without loss of generality, take $m=0$.
We denote $\mathbb R^- := (-\infty,0]$.  
Define
\[
f^-(x) := \frac{f(x)\mathbf 1_{\mathbb R^-}(x)}{p^{-}},
\qquad
p^{-} := \int_{\mathbb R^-} f(x)\,dx,
\]
and
for $\alpha \ge 0$, let
\[
G(\alpha)
:= \frac{\|f\|_\alpha}{\|f\,\mathbf 1_{\mathbb R^-}\|_\alpha}
= \left(
\frac{\int_{\mathbb R} f(x)^\alpha\,dx}{\int_{\mathbb R^-} f(x)^\alpha\,dx}
\right)^{1/\alpha}.
\]

\medskip

We claim that the function $\alpha \mapsto G(\alpha)$ is non-increasing on $(0,\infty)$.
It suffices to show that $\frac{d}{d\alpha}\log G(\alpha)\le 0$.  
We first write
\[
\log G(\alpha)
=
\frac{1}{\alpha}\log\!\int_{\mathbb R} f(x)^\alpha\,dx
-
\frac{1}{\alpha}\log\!\int_{\mathbb R^-} f(x)^\alpha\,dx.
\]
A direct differentiation yields
\begin{equation}\label{derivada}
\frac{d}{d\alpha}\log G(\alpha)
=
\frac{1}{\alpha^2}
\Bigg[
\alpha\,\frac{\int_{\mathbb R} f(x)^\alpha \log f(x)\,dx}{\int_{\mathbb R} f(x)^\alpha\,dx}
-
\log\!\int_{\mathbb R} f(x)^\alpha\,dx
-
\alpha\,\frac{\int_{\mathbb R^-} f(x)^\alpha \log f(x)\,dx}{\int_{\mathbb R^-} f(x)^\alpha\,dx}
+
\log\!\int_{\mathbb R^-} f(x)^\alpha\,dx
\Bigg].
\end{equation}
To interpret \eqref{derivada} in terms of differential entropies, define the probability densities
\begin{equation} \label{powered}
g_\alpha(x)
:=
\frac{f(x)^\alpha}{\int_{\mathbb R} f(y)^\alpha\,dy},
\qquad
g_\alpha^-(x)
:=
\frac{f(x)^\alpha \mathbf 1_{\mathbb R^-}(x)}{\int_{\mathbb R^-} f(y)^\alpha\,dy}.
\end{equation}
A straightforward computation shows that
\[
h(g_\alpha)
=
-\int_{\mathbb R} g_\alpha(x)\log g_\alpha(x)\,dx
=
-\alpha\,\frac{\int_{\mathbb R} f(x)^\alpha \log f(x)\,dx}{\int_{\mathbb R} f(x)^\alpha\,dx}
+
\log\!\int_{\mathbb R} f(x)^\alpha\,dx,
\]
and similarly,
\[
h(g_\alpha^-)
=
-\int_{\mathbb R^-} g_\alpha^-(x)\log g_\alpha^-(x)\,dx
=
-\alpha\,\frac{\int_{\mathbb R^-} f(x)^\alpha \log f(x)\,dx}{\int_{\mathbb R^-} f(x)^\alpha\,dx}
+
\log\!\int_{\mathbb R^-} f(x)^\alpha\,dx.
\]
Substituting into \eqref{derivada} yields the identity
\begin{equation} \label{Gder}
\frac{d}{d\alpha}\log G(\alpha)
=
-\frac{1}{\alpha^2}
\big(
h(g_\alpha)-h(g_\alpha^-)
\big).
\end{equation}

Since $f$ is log-concave, $f^\alpha$ is log-concave for all $\alpha> 0$, and hence $g_\alpha$ is a log-concave probability density on $\mathbb R$.  
Applying Theorem \ref{entropygrunbaumTh} to $g_\alpha$ and its restriction $g_\alpha^-$, we obtain
\[
h(g_\alpha^-)\le h(g_\alpha).
\]
Together with \eqref{Gder}, this implies
\[
\frac{d}{d\alpha}\log G(\alpha)\le 0,
\]
so that $G(\alpha)$ is non-increasing on $(0,\infty)$.
Consequently,
\[
G(\alpha)\le G(1)\qquad\text{for all }\alpha\ge 1,
\]
and 
\[
G(\alpha)\geq G(1)\qquad\text{for all }\alpha\leq 1,
\]
which is in both cases equivalent to
\[
h_\alpha(f^-)\le h_\alpha(f).
\]
\end{proof}


We now turn to the equality cases. We will need the following lemma:
\begin{lemma}[\cite{santalo}, Lemma 3]
\label{lem:affine-rigidity}
Let $\phi:\mathbb{R}\to \mathbb{R}\cup\{+\infty\}$ be a convex function such that $
\int_{\mathbb{R}} e^{-\phi(t)}\,dt < +\infty.$
Let
\[
U:=\{\phi<+\infty\}.
\]
Suppose that for some $x\in U$ and for almost every $y\in U$, one has
\[
\phi(x)=\phi(y)+(x-y)\phi'(y).
\]
Then for every $z\in U$, the function $\phi$ is affine on the segment $[x,z]$.
\end{lemma}

The next proposition gives the one-dimensional equality cases in 
\begin{equation} \label{matthieu}
\log\|g\|_\infty \le 1+\int_{\mathbb R} g\log g,
\end{equation}
which is established in \cite[Proof of Theorem 7]{fradelizi2008increasing}.

\begin{proposition}
\label{prop:fradelizi-equality-1d}
Let $g:\mathbb R\to[0,\infty)$ be a log-concave probability density, and write
$g=e^{-\phi},$ where $\phi:\mathbb R\to \mathbb R\cup\{+\infty\}$ is convex.
Then
\[
\log\|g\|_\infty = 1+\int_{\mathbb R} g\log g
\]
if and only if

\begin{equation} \label{formula}
g(t)=\frac{1}{\alpha+\beta}
\begin{cases}
e^{\frac{t-t_0}{\alpha}}, & t\le t_0,\\[1mm]
e^{-\frac{t-t_0}{\beta}}, & t\ge t_0,
\end{cases}
\qquad\text{for some }t_0 \in \mathbb R, (\alpha, \beta) \in \mathbb R_+^2 \setminus \{(0,0)\}.
\end{equation}
\end{proposition}
Before giving the proof, we note that \eqref{formula} includes the cases where $\alpha$ or $\beta$ are equal to $0$ (but not both). In other words, the equality cases are precisely the two one-sided exponential
families and the asymmetric double-exponential family.
\begin{proof}
Let $t_0$ be a point where $g$ attains its maximum, equivalently where $\phi$
attains its minimum.
Since $\|g\|_\infty=e^{-\phi(t_0)}$ and $
\int_{\mathbb R} g\log g
=
-\int_{\mathbb R}\phi(y)e^{-\phi(y)}\,dy,$ 
the equality
\[
\log\|g\|_\infty = 1+\int g\log g
\]
is equivalent to
\begin{equation}\label{eq:frad-equality}
\phi(t_0)=\int_{\mathbb R}\phi(y)e^{-\phi(y)}\,dy-1.
\end{equation}
Let
\[
U:=\{\phi<+\infty\},
\qquad
a:=\inf U,\qquad b:=\sup U.
\]
We proceed as in the proof of \cite[Theorem 7]{fradelizi2008increasing}. Convexity gives, for every $y\in U$,
\[
\phi(t_0)\ge \phi(y)+(t_0-y)\phi'(y).
\]
Multiplying by $e^{-\phi(y)}$ and integrating over $U$, we obtain
\begin{equation}\label{eq:frad-step1}
\phi(t_0)
\ge
\int_U \phi(y)e^{-\phi(y)}\,dy
+
\int_U (t_0-y)\phi'(y)e^{-\phi(y)}\,dy.
\end{equation}
An integration by parts yields
\[
\int_U (t_0-y)\phi'(y)e^{-\phi(y)}\,dy
=
\Big[-(t_0-y)e^{-\phi(y)}\Big]_{a}^{b}
-\int_U e^{-\phi(y)}\,dy.
\]
Since $\int_U e^{-\phi}=1$, this becomes
\begin{equation}\label{eq:frad-step2}
\int_U (t_0-y)\phi'(y)e^{-\phi(y)}\,dy
=
-(t_0-b)e^{-\phi(b)}+(t_0-a)e^{-\phi(a)}-1
\ge -1.
\end{equation}
Combining \eqref{eq:frad-step1} and \eqref{eq:frad-step2} gives
\[
\phi(t_0)\ge \int_U \phi(y)e^{-\phi(y)}\,dy-1
=
\int_{\mathbb R}\phi(y)e^{-\phi(y)}\,dy-1.
\]
This is precisely the inequality
\[
\log\|g\|_\infty \le 1+\int_{\mathbb R}g\log g.
\]
Moreover, equality in this inequality is equivalent to \eqref{eq:frad-equality}.
Assume now that equality holds in \eqref{eq:frad-equality}. Then equality must
hold both in \eqref{eq:frad-step1} and \eqref{eq:frad-step2}.
Equality in
\eqref{eq:frad-step1} implies
\[
\phi(t_0)=\phi(y)+(t_0-y)\phi'(y)
\qquad\text{for almost every }y\in U,
\]
because the integrand
\[
\phi(t_0)-\phi(y)-(t_0-y)\phi'(y)
\]
is nonnegative. By Lemma~\ref{lem:affine-rigidity}, $\phi$ is affine on every
segment $[t_0,z]\subset U$. In dimension one this means that $\phi$ is affine
on $[a,t_0]$ and on $[t_0,b]$. Equality in \eqref{eq:frad-step2} means precisely that the boundary term vanishes:
\begin{equation}\label{eq:boundary-vanish}
(t_0-a)e^{-\phi(a)}=0,
\qquad
(b-t_0)e^{-\phi(b)}=0.
\end{equation}
There are exactly two regimes:

\smallskip

\noindent\emph{Case (i):} Suppose first that $t_0$ coincides with one endpoint of the support. If
$t_0=a$, then the support is of the form $[t_0,\infty)$, and since $\phi$ is
affine on the whole support,
\[
\phi(t)=\phi(t_0)+\lambda(t-t_0)
\qquad (t\ge t_0)
\]
for some $\lambda>0$. Hence
\[
g(t)=\lambda e^{-\lambda(t-t_0)}\mathbf 1_{[t_0,\infty)}(t).
\]
This corresponds to the displayed formula with $\alpha=0$ and
$\beta=1/\lambda$. The case $t_0=b$ is symmetric and gives
\[
g(t)=\lambda e^{\lambda(t-t_0)}\mathbf 1_{(-\infty,t_0]}(t),
\]
which corresponds to the displayed formula with $\alpha=1/\lambda$ and
$\beta=0$.

\smallskip

\noindent\emph{Case (ii):} Suppose now that $t_0$ does not coincide with an endpoint of the support. Then
\eqref{eq:boundary-vanish} forces
\[
e^{-\phi(a)}=e^{-\phi(b)}=0.
\]
Since $\phi$ is affine on each side of $t_0$, this implies that $a=-\infty$ and
$b=+\infty$. Therefore
\[
\phi(t)=\phi(t_0)+\frac{t_0-t}{\alpha}
\quad (t\le t_0),
\qquad
\phi(t)=\phi(t_0)+\frac{t-t_0}{\beta}
\quad (t\ge t_0)
\]
for some $\alpha,\beta>0$. Hence
\[
g(t)=C
\begin{cases}
e^{\frac{t-t_0}{\alpha}}, & t\le t_0,\\[1mm]
e^{-\frac{t-t_0}{\beta}}, & t\ge t_0,
\end{cases}
\]
and normalization gives
\[
C=\frac{1}{\alpha+\beta}.
\]

Conversely, a direct computation shows that every density described above
satisfies
\[
\log\|g\|_\infty = 1+\int g\log g .
\]
This completes the proof.
\end{proof}




Finally we characterize the equality cases in Theorem \ref{thm:renyi-truncation}]:
\begin{proposition}
\label{cor:renyi-truncation-equality}
Let $f$ be a log-concave probability density and for $m \in \RL$, let 
\[
f_{m}^-(x):=\frac{f(x)\mathbf 1_{(-\infty,m]}(x)}{\int_{-\infty}^{m} f(y)\,dy}.
\]
Let also $\alpha\in(0,\infty]$. Then for any $m \in \RL$,
\[
h_\alpha(f_m^-)=h_\alpha(f)
\]
if and only if either $\int_{-\infty}^{m} f(y)\,dy =1$, or there exist $\lambda>0$ and $d\ge m$ such that
\[
f(x)=\lambda e^{\lambda(x-d)}\mathbf 1_{(-\infty,d]}(x).
\]
\end{proposition}

\begin{proof}
Consider first the differential entropy case $\alpha =1$ and define, as before, $p_m^- := \int_{-\infty}^{m} f(y)\,dy$. If $p_m^{-}=1$, then $f_m^-=f$ a.e., so equality is immediate. Assume now that $p_m^{-}<1$ and that
\[
h(f_m^-)=h(f).
\]
Let
\[
d:=\sup\operatorname{supp}(f).
\]
For $x\in[m,d)$, define
\[
F(x):=\int_{-\infty}^x f(y)\,dy,
\qquad
g_x(y):=\frac{f(y)\mathbf 1_{(-\infty,x]}(y)}{F(x)}.
\]
By the proof of Theorem~\ref{entropygrunbaumTh}, the equality $h(f_m^-)=h(f)$ forces equality in   
inequality~\eqref{matthieusineq}
for every truncated density $g_x$, $x\in[m,d)$. 
By Proposition~\ref{prop:fradelizi-equality-1d}, each $g_x$ must belong to one of the equality regimes described there. Since
\[
\operatorname{supp}(g_x)\subset(-\infty,x],
\]
the asymmetric double-exponential regime is impossible, and the only possible
one-sided regime is the exponential supported on a left half-line. Hence, for every $x\in[m,d)$, there exists $\lambda_x>0$ such that
\[
g_x(y)=\lambda_x e^{\lambda_x(y-x)}\mathbf 1_{(-\infty,x]}(y).
\]
Equivalently,
\[
f(y)=F(x)\lambda_x e^{\lambda_x(y-x)}
\qquad (y\le x).
\]
Comparing these identities for different values of $x$, we conclude that $\lambda_x$ does not depend on $x$. It follows that $f$ must be a one-sided exponential supported on a left half-line, namely
\[
f(x)=\lambda e^{\lambda(x-d)}\mathbf 1_{(-\infty,d]}(x)
\]
for some $\lambda>0$ and some $d\ge m$. The converse is immediate by the memoryless property of the exponential density.

Now let $0<\alpha \neq 1$. By translation invariance, we may assume $m=0$. Let
\[
p^{-}:=\int_{-\infty}^0 f(x)\,dx,
\qquad
f^-(x):=\frac{f(x)\mathbf 1_{(-\infty,0]}(x)}{p^{-}}
\quad \text{and} \quad  
G(\alpha):=\frac{\|f\|_\alpha}{\|f\mathbf 1_{(-\infty,0]}\|_\alpha}.
\]

Suppose $h_\alpha(f^-)=h_\alpha(f)$. 
As in the proof of
Theorem~\ref{thm:renyi-truncation}, we deduce that $G(\alpha)=G(1)$. Since
$G$ is non-increasing on $(0,\infty)$, the function $G$
must therefore be constant on $[\min\{\alpha,1\},\max\{\alpha,1\}]$. By \eqref{Gder}, we have equality in $h(g_{\alpha}^-) = h(g_{\alpha})$, where $g_{\alpha}$ is defined as in \eqref{powered}. By the differential entropy case, we conclude that
either $\operatorname{supp}(g_{\alpha})\subset(-\infty,0]$, i.e.\ $p^{-}=1$, or else $g_{\alpha}$ is a one-sided exponential density. Since $g_\alpha\propto f^\alpha$, the latter implies that
\[
f(x)=\lambda e^{\lambda(x-d)}\mathbf 1_{(-\infty,d]}(x)
\]
for some $\lambda>0$ and $d\ge0$. The converse is again immediate by the memoryless property of the exponential density.  The case \(\alpha=\infty\) was already proved in Proposition \ref{min-entroy}.
\end{proof}

\subsection{The centering assumption in higher dimensions}
\label{counterexamplesec}

Finally, we record the example mentioned in the Introduction, which shows that, already for min-entropy in dimension
two, one cannot expect an analogous statement without a centering assumption on the
cutting hyperplane.

\begin{example} \label{counterexample}
Let \(a\geq1\), and set
\[
H=\{x\in\mathbb R^2:x_1+x_2=1\},\qquad
H^+=\{x\in\mathbb R^2:x_1+x_2\ge 1\},
\qquad
H^-=\{x\in\mathbb R^2:x_1+x_2\le 1\}.
\]
Consider $f(x) = (1-\frac{x_1}{a}-\frac{x_2}{a})_+, x_1,x_2 \geq 0.$ Then 
$$
\int_{\mathbb{R}^2} f = \frac{a^2}{6},\hspace{0.2cm} \max_{x \in H^+}f(x) = 1-\frac{1}{a}.
$$
On the other hand, direct computation shows 
$$
\int_{H^-}f = \frac{1}{2}-\frac{1}{3a}
$$
and thus
$$
\int_{H^+}f = \frac{a^2}{6}-\frac12+\frac1{3a}.
$$
The inequality 
\begin{equation} \label{ineqinexample}
\frac{\max_{H^+}f}{\max f} \geq \frac{\int_{H^+}f}{\int_{\mathbb{R}^2} f}
\end{equation}
is equivalent to 
$$
1 \leq \frac3{a}-\frac2{a^2}
$$
which is 
$$
(a-1)(a-2) \leq 0,
$$
i.e. satisfied for $a \in [1,2].$
One verifies that
$$
\frac{\int_{\mathbb{R}^2} xf(x)dx}{\int_{\mathbb{R}^2} f(x)dx} = (a/4,a/4).
$$

Hence $\int \frac{xf(x)}{\int f} dx \in H$ if and only if $a=2$. Consequently, taking any \(a>2\), the barycenter of \(f\) does not lie on \(H\), and the inequality
\eqref{ineqinexample} fails. This shows that, in dimension two, the condition that the cutting hyperplane pass through the barycenter cannot simply be omitted in the min-entropy truncation inequality.

\end{example}

\section{Proof of Gr\"unbaum inequalities for differential and min--entropy}
\label{grunbaumSec}

Let $f$ be a zero-mean log-concave probability density on $\mathbb R$,
and 
\[
        f^+(x):=
        \frac{f(x)\mathbf 1_{\{x\ge 0\}}}{\int_0^\infty f(y)\,dy}.
\]
In this section, we would like to understand, for every $\alpha>0$, what is the best 
constant $C_\alpha<\infty$ such that
\[
        h_\alpha(f^+) \ge h_\alpha(f)-C_\alpha .
\]
Such a constant exists by monotonicity of R{\'e}nyi entropy together with the reverse comparison for log-concave densities~\cite{fradelizilimadiman} and Theorem \ref{thm:sharp-reverse-grunbaum-shannon} for differential entropy that we prove later in this section. 

Since $h_0(f)=\log |\operatorname{supp} f|$, no uniform finite constant should
hold at $\alpha=0$. Thus, we must have
\[
        C_\alpha\longrightarrow +\infty
        \qquad\text{as }\alpha\downarrow 0.
\]

Moreover, as we will see later in this section, the extremizers are different for $\mathrm{min}$--entropy and for differential entropy. This indicates that a unified reverse result is more subtle to obtain and we expect the extremizers to be some interpolating family of densities. Below we establish sharp reverse inequalities for the limiting cases $\alpha=1,\infty$. We will need the following lemmas.

\begin{lemma} [Lemma of Section 2.1. in \cite{milmanisotropic}]
\label{lem:moment-lower-decreasing}
Let $g:[0,\infty)\to[0,\infty)$ be integrable. 
\[
        \int_0^\infty t g(t)\,dt
        \ge
        \frac{\left(\int_0^\infty g(t)\,dt\right)^2}{2\sup g(x)}
\]
with equality if and only if $g$ is constant on an interval $[0,a]$ for some $a \geq 0.$
\end{lemma}

\begin{lemma}[Borell \cite{borell}]
\label{lem:moment-upper-logconcave}
Let $g:[0,\infty)\to[0,\infty)$ be integrable and log-concave. 
Then
\[
        \int_0^\infty t g(t)\,dt
        \le
        \frac{\left(\int_0^\infty g(t)\,dt\right)^2}{g(0)}
\]
with equality if and only if $g$ is exponential on $[0,+\infty).$
\end{lemma}

We first consider the limiting case $\alpha=\infty$.
\begin{theorem} \label{minentropyreverseTh}
Let $f:\mathbb R\to[0,\infty)$ be a log-concave probability density such that
\[
        \int_{\mathbb R} x f(x)\,dx=0.
\]
Then
\[
        h_\infty(f^+)
        \ge
        h_\infty(f)-\log(1+\sqrt2).
\]
Equivalently,
\begin{equation} \label{sqrt2equality}
        \sup_{\mathbb R_+} f(x)
        \le
        (1+\sqrt2)
        \left(\sup_{x\in\mathbb R} f(x)\right)
        \int_0^\infty f(x)\,dx .
\end{equation}
There is equality if and only if $f$ is constant on a finite interval $[-a,0], a \geq 0,$ and exponential on $\mathbb R_+$.
\end{theorem}

\begin{proof}

Set
\[
        p_{+}:=\int_0^\infty f(x)\,dx
        \quad \text{and } \quad
        p_{-}:=\int_{-\infty}^0 f(x)\,dx.
\]

Note first that if $p_{-}\leq \sqrt{2} p_{+}$ the inequality follows immediately since 
\begin{equation} \label{qleqsqrtp}
    \sup_{\mathbb R_+} f(x)=(p_{+}+p_{-})\sup_{\mathbb R_+} f(x) \leq (1+\sqrt2)p_{+} \sup_{\mathbb R_+} f(x) \leq (1+\sqrt2)p_{+} \sup_{x\in \mathbb R} f(x).
\end{equation}

If $\sup_{\mathbb R_+} f(x) = \sup_{x\in \mathbb R} f(x)$, then we apply Lemma \ref{lem:moment-lower-decreasing} to $g(t) = f(-t)\mathbf 1_{\mathbb R_+}$ and Lemma \ref{lem:moment-upper-logconcave} to $f(t) \mathbf 1_{\mathbb R_+}$ and obtain 
$$
\frac{(p_{-})^{2}}{2f(0)}\leq \frac{(p_{+})^{2}}{f(0)},
$$
whence 
\begin{equation} \label{asabove}
    p_{-} \leq \sqrt{2}p_{+}
\end{equation}
and the claimed inequality holds by \eqref{qleqsqrtp}. 

So it suffices to prove the inequality when $\sup_{x \in \mathbb R_-}f(x) = \sup_{x \in \mathbb R}f(x)$ and $p_{-} > \sqrt2 p_{+}.$ Then by Lemmas \ref{lem:moment-lower-decreasing} and \ref{lem:moment-upper-logconcave} again, 
\begin{equation} \label{lemmasqgreater}
\frac{(p_{-})^{2}}{2\sup_{x \in \mathbb R_-}f(x)}\leq \frac{(p_{+})^{2}}{f(0)}.
\end{equation}
But then, since in this case $\sup_{\mathbb R_+} f(x) = f(0)$
$$
\frac{\sup_{\mathbb R_+} f(x)}{\sup_{x \in \mathbb R}f(x) \int_0^{\infty}f(x)dx}\leq \frac{2p_{+}}{(p_{-})^{2}} = (1+\frac{p_{-}}{p_{+}})\frac{2(p_{+})^{2}}{(p_{-})^{2}} \leq 1+\sqrt{2}
$$
since $\frac{p_{+}}{p_{-}} < \frac{1}{\sqrt{2}}.$ 

Now suppose there is equality in \eqref{sqrt2equality}. By \eqref{qleqsqrtp}, $p_{-}<\sqrt2 p_{+} $ is impossible as otherwise the inequality would be strict. If $\sup_{\mathbb R_+} f(x) = \sup_{x\in \mathbb R} f(x)$, we must have $p_{-} = \sqrt{2}p_{+}$ as in \eqref{asabove}. On the other hand, if $p_{-} \geq \sqrt 2  p_{+}$ and $\sup_{\mathbb R_-} f(x) = \sup_{x\in \mathbb R} f(x)$ we must have $p_{-} = \sqrt{2}p_{+}$, as otherwise we would have strict inequality by \eqref{lemmasqgreater}. 

So it suffices to characterize the equality case if $p_{-} = \sqrt 2 p_{+}$. But then 
equality in \eqref{sqrt2equality} reads 
$$
\sup_{x\in \mathbb R_+} f(x) = \max_{\mathbb R}f(x).
$$ 
In this case there must be equality in both the applications of Lemma \ref{lem:moment-lower-decreasing} to $g(t)$ and Lemma \ref{lem:moment-upper-logconcave}, which yields the claimed characterization. 
\end{proof}

Next we prove the $\alpha = 1$ case of Theorem \ref{secondThIntro}. In the proof we use Proposition \ref{prop:normalized-variational-inequality}, which is proved at the end of this section. 
\begin{theorem}
\label{thm:sharp-reverse-grunbaum-shannon}
Let \(X\) be a real-valued random variable with log-concave density \(g\), and assume that
\[
        \int_{\mathbb R} xg(x)\,dx=0.
\]
Then
\[
        h(g^+)
        \ge
        h(g)-\frac{e}{e-1}H_2(1/e).
\]
Equality holds if and only if $g$ is the exponential density
\[
        g(x)=\lambda e^{\lambda(x-1/\lambda)}
        \mathbf 1_{(-\infty,1/\lambda]}(x),
        \qquad \lambda>0.
\]
\end{theorem}


\begin{proof}
 
Let
\[
        p_+ := \int_0^\infty g(x)\,dx,
        \quad p_- = 1-p_+ \quad \text{and }
        g^+(x):=\frac{g(x)\mathbf 1_{\{x\ge0\}}}{p_+}, g^-(x):=\frac{g(x)\mathbf 1_{\{x<0\}}}{p_-}
\]
be the densities of \(X^+\) and \(X^-\). Let also
\[
        m_+ := \int_0^\infty x g^+(x)\,dx \quad \text{and } m_- = \int_{-\infty}^0 |x| g^-(x)\,dx.
\]
Therefore $p_+m_+ = p_-m_-$. We have
$$
h(g) = H_2(p_+) + p_+h(g^+)+p_-h(g^-)
$$
and thus 
\begin{align}
    h(g)-h(g^+) &= H_2(p_+) + p_-\bigl(h(g^-) - h(g^+)\bigr).
\end{align}
Since the negative exponential maximizes entropy under fixed mean on $\mathbb R_-$ \cite[Example 12.2.5]{cover:book} (as can be seen by $D(g^-\|\mathrm{exp}_-(\frac{1}{m_-})) \geq 0$), we obtain after applying the second part of Proposition \ref{prop:normalized-variational-inequality} to $g^+$,
\begin{align}
    h(g)-h(g^+) &\leq H_2(p_+) + p_-\bigl( 1 + \log{{m_-}} - \log{m_+}-2\log{(e-1)} +\frac{1}{e-1}\bigr) \\
    &= H_2(p_+) + p_-\bigl( 1 -\log{p_-} + \log{p_+}-2\log{(e-1)} +\frac{1}{e-1}\bigr) \\ \label{maximizingexpr}
    &= 2H_2(p_-) + p_-\bigl( 1 -2\log{(e-1)} +\frac{1}{e-1}\bigr) + \log{(1-p_-)}.
\end{align}
Considering the function $x \mapsto -2x\log{x} -2(1-x)\log{(1-x)} + cx + \log{(1-x)}, x \in [0,1],$ where $c = 1-2\log{(e-1)} +\frac{1}{e-1},$ and setting its derivative to zero, we see that the maximum is achieved at $x = 1/e$. Therefore, \eqref{maximizingexpr} is maximized for $p_- = 1/e$ and therefore
\begin{align}
h(g)-h(g^+) &\leq 2H_2(1/e) + 1/e\bigl( 1 -2\log{(e-1)} +\frac{1}{e-1}\bigr) + \log{(1-1/e)} \\
&= \frac{e}{e-1}H_2(1/e)
\end{align}
after a short calculation, with equality if and only if $p_- = 1/e$. The equality case follows from the equality case in the functional form of Gr{\"u}nbaum's inequality for log-concave densities, see for example~\cite[Theorem 1.4.]{fradelizi2026grunbaum}.
\end{proof}

Next we establish the lower bound on the entropy of a density on the positive reals that was used in the proof of Theorem \ref{thm:sharp-reverse-grunbaum-shannon}, but which may also be of independent interest. It implies in particular that the entropy of the restriction of a centered log-concave density on $\mathbb R_+$ is minimized by an exponential; this is analogous to \cite{MNR} where the entropy among log-concave densities with fixed variance is minimized for an exponential.
\begin{proposition}
\label{prop:normalized-variational-inequality}
Let \(f\) be a log-concave probability density on \(\mathbb R_+\) such that $f(0)>0$ and let $\kappa:=(\log f)'_+(0)$.

If 
$$
        \sqrt{f(0)}\ge \kappa\sqrt{\int_0^\infty tf(t) \ dt},
$$
then
\begin{equation} \label{variationalineq}
        h(f)\ge \log\Biggl({\int_0^\infty tf(t)\dd t}\Biggr)+2\log(e-1)-\frac1{e-1},
\end{equation}
where the right hand side is the entropy of the truncated exponential $f^*(x) = \frac1{\lambda(e-1)^2}e^{\frac{x}{\lambda (e-1)}}\1_{(0,\lambda (e-1))}(x), $ with $\lambda = \int_0^\infty tf(t)\dd t.$

In particular, if $g$ is a centered log-concave density on \(\mathbb R\), i.e., 
$
\int_{\mathbb R} t g(t)\dd t=0,
$
then \eqref{variationalineq} holds with $g^+$ in place of $f$, where
$
g^+(t):=\frac{g(t)}{\int_0^\infty g(x)\dd x}\1_{\{t>0\}}.
$ 
\end{proposition}

The proof is divided in several lemmas.

\begin{lemma} \label{lemmasatisfiescond}
Let $f$ be a centered log-concave density. If
$$
\frac{d}{dx} \log{f(x)}\Bigg|_{0^+}= \bigl(\log{f}\bigr)_+'(0) \geq 0,
$$
then 

\[
f(0)\ge {\Bigl(\bigl(\log{f}\bigr)_+'(0)\Bigr)^2}{\int_0^\infty tf(t) \ dt}.
\]
\end{lemma}

\begin{proof}

Since \(f\) is log-concave, for \(t<0\),

\[
\log f(t)
\le
\log f(0)+\bigl(\log{f}\bigr)_+'(0)t.
\]

Because \(f\) is centered,

\[
\int_{-\infty}^0 (-t){f(t)}\dd t
=
\int_0^\infty t{f(t)}\dd t
.
\]
Since \(\bigl(\log{f}\bigr)_+'(0) \geq 0\),

\[
\int_0^\infty t{f(t)}\dd t \leq f(0)\int_{-\infty}^0 (-t)e^{\bigl(\log{f}\bigr)_+'(0)t}\dd t
=
\frac{f(0)}{{\bigl(\bigl(\log{f}\bigr)_+'(0)\bigr)^{2}}
}\]
and the claim follows.
\end{proof}

\begin{proposition}
\label{prop:existence-reduction}
Let \(R>1\), and let \(\mathcal F_R\) be the class of log-concave
probability densities \(f=e^u\) supported on \([0,R]\) such that
\[
\int_0^R f(x)\,dx=1,
\qquad
\int_0^R xf(x)\,dx=1,
\]
and such that
\[
f(0)>0
\qquad\text{and}\qquad
\sqrt{f(0)}\geq u'_+(0).
\]
Then \(\mathcal F_R\) is compact with respect to weak convergence of the associated probability measures,
and entropy is a continuous functional on \(\mathcal F_R\). Consequently, the entropy functional attains its minimum on
\(\mathcal F_R\). Moreover, if \(f\in\mathcal F_R\) is a minimizer,
then \(\log f\) is affine on at most three intervals.
\end{proposition}

\begin{proof}

Let $(\mu_n)_n$ be a sequence of probability measures in $\mathcal F_R $  which converges weakly to a measure $\mu$. Denote by $f_n$  the density of $\mu_n$. 
First, we observe that the set of log-concave probability measures on a compact set is compact with respect to weak convergence, hence $\mu$ is log-concave on $[0,R]$. Thus $\mu$ is either a Dirac measure or has a log-concave density.
The boundary condition at the origin is equivalent to
\[
f_n(x)\leq f_n(0)e^{\sqrt{f_n(0)}\,x},
\qquad x\in[0,R].
\]
The normalization $\int_0^R f_n=1$ prevents $f_n(0)$ from tending to
zero. Moreover, the condition $\int_0^R xf_n(x)\,dx=1$ forces a
uniformly positive amount of mass to lie away from the origin, and
log-concavity then prevents $f_n(0)$ from tending to infinity.
Hence $(f_n)_n$, is uniformly bounded. It follows that the weak limit $\mu$ cannot be a Dirac mass, and
therefore has a log-concave density $f$ with mean $1$. After passing
to a subsequence, assume that $f_n(0)\to l>0$. By
\cite[Fact~2.5]{corderofradelizi}, $f_n\to f$ uniformly on compact
subsets of $\operatorname{int}(\operatorname{supp}f)$. Passing to the
limit in the upper bound above and in the corresponding lower bound
given by log-concavity, and then letting $x\downarrow0$, yields
$l=f(0)$. Hence $f$ also satisfies the boundary condition at the origin, so
$f\in\mathcal F_R$. Thus $\mathcal F_R$ is weakly closed, hence
compact. Finally, the uniform bound on $(f_n)_n$ and dominated
convergence give
\[
h(f_n)\longrightarrow h(f).
\]

We now turn to the structural characterization of the minimizers.
Let \(f\in\mathcal F_R\) be a minimizer and write \(f=e^{-V}\), with
\(V\) convex. 
Assume, by contradiction, that \(f\) is not
three-piece log-affine. By the one-dimensional degree-of-freedom
construction of
\cite[Proof of Proposition~2]{fradeliziguedon}, there are enough
independent admissible perturbation directions, all vanishing on an
initial interval. Taking a non-trivial linear combination, we may
choose a bounded continuous function \(W\), still vanishing near
\(0\), such that, for all sufficiently small \(t>0\),
\[
f_\pm:=f(1\pm tW)
\]
are log-concave and
\[
\int_0^R Wf=0,
\qquad
\int_0^R xWf=0.
\]
The two identities preserve mass and mean, while the fact that \(W\)
vanishes near \(0\) preserves the endpoint condition. Hence
\[
f_+,f_-\in\mathcal F_R.
\]

Finally,
\[
f=\frac{f_++f_-}{2}.
\]
Since \(x\mapsto x\log x\) is strictly convex and \(W\neq0\),
\[
h(f)>
\frac12h(f_+)+\frac12h(f_-).
\]
Thus at least one of \(f_+\) and \(f_-\) has entropy strictly smaller
than \(h(f)\), contradicting the minimality of \(f\). Therefore
\(\log f\) is affine on at most three intervals.
\end{proof}

We will use the following optimization lemma to show that in fact a minimizer must be one-piece log-affine. 
\begin{lemma}[Fritz--John necessary conditions; Proposition 3.3.5 in \cite{bertsekas}] \label{FJ}
Let 

\[
F,G_1,\dots,G_r,C_1,\dots,C_s\in C^1(\mathbb R^n).
\]
Suppose \(z_0\in \mathbb{R}^n\) is a local minimizer of \(F\) subject to

\[
G_i(z)=0,\qquad i=1,\dots,r,
\]
and

\[
C_j(z)\le0,\qquad j=1,\dots,s.
\]
Then there exist multipliers

\[
\tau\ge0,\qquad
\alpha_1,\dots,\alpha_r\in\mathbb R,
\qquad
\gamma_1,\dots,\gamma_s\ge0,
\]
not all zero, such that

\[
\tau\nabla F(z_0)
+
\sum_{i=1}^r \alpha_i\nabla G_i(z_0)
+
\sum_{j=1}^s \gamma_j\nabla C_j(z_0)=0.
\]
Moreover,

\[
\gamma_j C_j(z_0)=0,
\qquad j=1,\dots,s.
\]
\end{lemma}

We are finally ready to give the proof of Proposition \ref{prop:normalized-variational-inequality}.
\begin{proof} [Proof of Proposition \ref{prop:normalized-variational-inequality}.]
    
First note that, by approximating $f$ with compactly supported densities, we may assume that $f$ has a compact support. Then, we observe that, by considering the mean-normalized density  
$ 
\mu f(\mu x), \quad x>0,
$
where $\mu = {\int_0^{\infty}f(x)xdx},$
and using the scaling property of differential entropy, it suffices to prove the normalized inequality
\[
h(f)\ge 2\log(e-1)-\frac1{e-1}
\]
for any log-concave density $f$ on $[0,R]$ with $\int_{\mathbb R_+}f(x)xdx = 1$ and $\sqrt{f(0)} \geq \kappa$. 
By Proposition \ref{prop:existence-reduction}, the entropy functional attains its minimum on
\(\mathcal F_R\) at some density \(f_R\), and every such minimizer is
three-piece log-affine. Therefore, it suffices to prove the normalized
inequality for a three-piece log-affine minimizer. Renaming this
minimizer \(f\), it can be parametrized as
\[
f(x)=e^{u(x)}\1_{(0,R)}(x),
\]
where \(R\in(0,\infty)\) and
\[
u(x)=a+\kappa x-\lambda_1(x-r_1)_+-\lambda_2(x-r_2)_+.
\]
Here
$
0<r_1<r_2<R$ and 
$\lambda_1,\lambda_2\ge0.
$
The slopes of \(u\) are $\kappa,
\kappa-\lambda_1$ and $\kappa-\lambda_1-\lambda_2.$
In particular, concavity is exactly the condition
\[
\lambda_1,\lambda_2\ge0.
\]
There is a change of slope at a point \(r_j\) precisely when
\[
\lambda_j>0.
\]
We will use the Fritz-John necessary conditions to show that if a density of this form is a local minimizer of the entropy functional, then it must be the case that $\lambda_j = 0$ for $j=1,2$ and thus any local minimizer is one-piece log-affine. 

We are minimizing locally the entropy 
\[
h(f)=-\int_0^R u(x)e^{u(x)}\dd x.
\]
in the variables $(a, \kappa, \lambda_1,\lambda_2,r_1, r_2) \in \mathbb R^6$ subject to the constraints
\[
\int_0^R e^{u(x)}\dd x=1,
\quad 
\int_0^R x e^{u(x)}\dd x=1 \quad \text{and } \quad \kappa-e^{\frac{a}{2}}\le0.
\]
We will assume for contradiction that at a local minimizer $\lambda_i, r_i >0$  for $i=1,2$ and hence by the complementary slackness property \cite[Proposition 3.3.5, (iv)]{bertsekas} the corresponding multipliers for the conditions $\lambda_1,\lambda_2 \geq 0$ and $0\leq r_1 \leq r_2 \leq R$ are $0$, since the constraints are inactive. Alternatively, we may consider the corresponding functions extended in a continuously differentiable way on the whole $\mathbb R^6$. 

The partial derivatives of \(u\) are

\[
\frac{\partial u}{\partial a}=1,
\qquad
\frac{\partial u}{\partial \kappa}=x, \quad
\frac{\partial u}{\partial \lambda_j}=-(x-r_j)_+
\text{ and for }
 x\neq r_j, \quad 
\frac{\partial u}{\partial r_j}
=
\lambda_j\1_{\{x>r_j\}}.
\]

Suppose first that \(f\) is a local minimizer with two changes of slope. Thus
\[
\lambda_1,\lambda_2>0
\qquad \text{and } \quad
0<r_1<r_2<R.
\]


By the Fritz--John necessary conditions, Lemma \ref{FJ}, there exist multipliers
$\tau\ge0,$ 
$\alpha,\beta\in\mathbb R,$ 
$\gamma\ge0,$
not all zero, such that the gradient of

\[
\tau h(f)
+
\alpha\left(\int_0^R f-1\right)
+
\beta\left(\int_0^R x f-1\right)
+
\gamma(\kappa-e^{\frac{a}2})
\]
vanishes. 

We first show that
$
\tau>0.$
Assume \(\tau=0\). Since \(\kappa-e^{\frac{a}{2}}\) does not depend on \(r_2\) or \(\lambda_2\), differentiating with respect to \(r_2\) and \(\lambda_2\) gives

\[
\int_{r_2}^R (\alpha+\beta x)f(x)\dd x=0
\quad \text{ and } \quad
\int_{r_2}^R (x-r_2)(\alpha+\beta x)f(x)\dd x=0.
\]
On \((r_2,R)\), the function \(\alpha+\beta x\) is affine. Writing 

\[
\alpha+\beta x=A+B(x-r_2),
\]
we have

\[
\int_{r_2}^R (\alpha+\beta x)^2f(x)\dd x
=
A\int_{r_2}^R (\alpha+\beta x)f(x)\dd x
+
B\int_{r_2}^R (x-r_2)(\alpha+\beta x)f(x)\dd x
=
0.
\]
Since \(f>0\) on \((r_2,R)\), it follows that
$ 
\alpha+\beta x=0
$ on $(r_2,R)$ and hence

\[
\alpha=\beta=0.
\]
Then differentiating with respect to \(a\) gives

\[
0=\gamma\frac{\partial}{\partial a}(\kappa-e^\frac{a}{2})
=
-\frac12\gamma e^{\frac{a}{2}}.
\]
Thus \(\gamma=0\), contradicting that the multipliers are not all zero. Therefore
$
\tau>0.
$
Dividing all multipliers by \(\tau\), we may assume \(\tau=1\).

Denote as shorthand

\[
W(x):=-(1+u(x))+\alpha+\beta x.
\]
Differentiating the Lagrangian with respect to \(\lambda_j\) gives

\begin{equation}
\int_{r_j}^R (x-r_j)W(x)f(x)\dd x=0.
\label{eq:danverm1}
\end{equation}

Differentiating with respect to \(r_j\) gives

\[
\lambda_j\int_{r_j}^R W(x)f(x)\dd x = 0.
\]
Since at \(x = r_j\) there is a change of slope, we must have \(\lambda_j>0\). Hence

\begin{equation}
\int_{r_j}^R W(x)f(x)\dd x=0.
\label{eq:danverm2}
\end{equation}

By \ref{eq:danverm1} and \ref{eq:danverm2} for $j=2$, we get

\[
\int_{r_2}^R Wf=0,
\qquad \text{and }
\int_{r_2}^R (x-r_2)Wf=0.
\]
On \((r_2,R)\), \(u\) is affine, so \(W\) is affine. Writing
\(
W(x)=A+B(x-r_2),
\)
we have

\[
\int_{r_2}^R W(x)^2f(x)\dd x
=
A\int_{r_2}^R Wf
+
B\int_{r_2}^R (x-r_2)Wf
=
0.
\]
Since \(f>0\) on \((r_2,R)\),
$
W=0$
on $(r_2,R).$
Thus
$W(r_2+)=0$ and also $W'_+(r_2)=0.$
But

\[
W'(x)=\beta-u'(x),
\]
which yields

\[
\beta -(\kappa -\lambda_1) = W'_-(r_2)=-\lambda_2.
\]
Therefore, on the middle interval \((r_1,r_2)\),
$W'(x) = -\lambda_2$ and so
\[
W(x)=W(r_2) - \int_{x}^{r_2}W'(t)dt = 0 - \int_{x}^{r_2}-\lambda_2dt = \lambda_2(r_2-x)>0.
\]
Also \(W=0\) on \((r_2,R)\). Hence

\[
\int_{r_1}^R W(x)f(x)\dd x
=
\int_{r_1}^{r_2} W(x)f(x)\dd x
>0.
\]
This contradicts \ref{eq:danverm2}
 with \(j=1\). Therefore a local minimizer cannot have two changes of slope.

Now suppose \(f\) is a local minimizer with exactly one change of slope:

\[
u(x)=a+\kappa x-\lambda(x-r)_+,
\qquad
\lambda>0,
\qquad
0<r<R.
\]
The same argument as in the previous case gives multipliers, with  nonzero multiplier for the entropy term, such that the gradient of the Lagrangian vanishes and hence a function

\[
W(x)=-(1+u(x))+\alpha+\beta x
\]
such that

\[
\int_r^R Wf=0,
\quad \text{and} \quad
\int_r^R (x-r)Wf=0.
\]
Since \(W\) is affine on \((r,R)\), the same square argument gives
$W=0$ on $(r,R).$
Thus
\[
W(r+)=0,
\quad \text{and} \quad
W'_+(r)=0.
\]
Differentiating $u$ we get
\[
W'_-(r)=-\lambda.
\]
Therefore, on \((0,r)\),

\[
W(x)=\lambda(r-x)>0.
\]
Now we differentiate the Lagrangian with respect to \(\kappa\). Since

\[
\frac{\partial u}{\partial \kappa}=x \quad \text{and} \quad \frac{\partial}{\partial\kappa}(\kappa-e^{\frac{a}2})=1,
\]
we get

\begin{equation} \label{secondcontr}
\int_0^R xW(x)f(x)\dd x+\gamma = 0.
\end{equation}
But \(W>0\) on \((0,r)\), \(W=0\) on \((r,R)\), and \(f>0\). Since also $\gamma\geq 0$

\[
\int_0^R xW(x)f(x)\dd x + \gamma
\geq
\int_0^r xW(x)f(x)\dd x
>0,
\]
contradicting \eqref{secondcontr}.

Thus a local minimizer cannot have one change of slope either.
Hence \(u\) is affine on its support, and

\[
f(x)=Ce^{\kappa x}\1_{(0,R)}(x),
\]
i.e., every local minimizer is one-piece log-affine. It remains to verify the desired inequality for such a density. Since $f$ is a mean-one density we must have 
\[
C=\frac{\kappa}{e^{\kappa R}-1},\quad \text{and} \quad
\frac{e^{\kappa R}(\kappa R-1)+1}{\kappa(e^{\kappa R}-1)} = 1.
\]
Writing \(t=\kappa R\),
we have
\[
h(f)=-\int_0^R f(x)\log f(x)\ dx
=
\log\Biggl(\frac{(e^t-1)^2}{e^t(t-1)+1}\Biggr)
-\frac{e^t(t-1)+1}{e^t-1}.
\]
The condition \((\log f)'(0)\le \sqrt{f(0)}\) is \(\kappa\le\sqrt C\), which is equivalent to $t\leq 1$. Differentiating the above expression for \(h\) with respect to \(t\), we get
\[
\frac{dh(f)}{dt}=
\frac{e^t(e^t t^2-(e^t-1)^2)}
{(e^t-1)^2(e^t(t-1)+1)}\le0,
\]
since \((e^t-1)^2\ge e^t t^2\). Therefore 
\[
h(f) \geq 2\log(e-1)-\frac1{e-1},
\]
which is the desired inequality.

    The statement for centered densities follows immediately from Lemma \ref{lemmasatisfiescond}.
\end{proof}

 \newpage 
\section*{Appendix}

\subsection*{Basic information theoretic properties}

In this appendix we summarize a few properties of entropy and mutual information. 

We recall here that the differential entropy of a vector $X$ with density $f$ on $\mathbb{R}^n$ is 
$$
h(X) = h(f) := -\int_{\mathbb{R}^n}{f(x)\log{f(x)} dx} \quad \in [-\infty, \infty].
$$
A change of variables in the integral shows that the differential entropy satisfies, for any invertible matrix $A \in \mathbb{R}^{n \times n}$, the scaling property
\begin{equation}
    h(AX) = h(X) + \log{|\det{(A)}|}.
\end{equation}

If $X \in \mathbb{R}^n,Y \in \mathbb{R}^m$ are two random variables (resp. vectors) with joint density $f_{X,Y}$ on $\mathbb{R}^{m+n}$ then the conditional density of $X$ given $Y$ exists a.s. and is given for any $y \in \mathbb{R}^m$ by 
$$
f_{X|Y}(x|y) := \frac{f_{X,Y}(x,y)}{f_Y(y)},
$$
where $f_Y(y) = \int{f_{X,Y}(x,y) dx}$ is the marginal of $Y$.
Then the joint and conditional entropies are defined respectively as 
$$
h(X,Y) := -\int_{\RL^{m+n}}{f(x,y)\log{f(x,y)}dxdy}
$$
and
$$
h(X|Y) := -\int_{\RL^{n+m}}{f_{X,Y}(x,y)\log{f_{X|Y}(x|y)}dxdy}.
$$
The conditional entropy can also be re-written as 
\begin{align} \nonumber
h(X|Y) &= -\int_{\RL^{n+m}}{f_{X,Y}(x,y)\log{f_{X|Y}(x|y)}dxdy} \\ \label{condHexpression}
&= -\int_{\RL^{m}} {f_Y(y)\int_{\RL^n}{f_{X|Y}(x|y)\log{f_{X|Y}(x|y)}dx}dy} \\ \label{condHaverage}
&= \int_{\RL^m}{f_Y(y)h(X|Y=y)}dy, 
\end{align}
where we denote, for $y \in \RL^m$, $h(X|Y=y) := h\bigl(f_{X|Y}(\cdot|y)\bigr).$ More generally, if $Y$ does not have a density, \eqref{condHaverage} can be replaced by 
\begin{equation} 
h(X|Y) = \int{h(X|Y=y) d\mu_Y(y)}
\end{equation}
where $\mu_Y$ stands for the law of $Y$.

The chain rule for entropy \cite{cover:book} is 
\begin{equation} \label{chainrule}
h(X,Y) = h(Y) + h(X|Y) = h(X) + h(Y|X).
\end{equation}

If $(X,Y)$ has joint law $\mu_{X,Y}$ and $X,Y$ have marginal laws $\mu_X$ and $\mu_Y$ respectively, the {\em mutual information} between $X$ and $Y$ is 
\begin{equation} \label{mutualinfodef}
I(X;Y) := \int d\mu_{X,Y}\log{\frac{d\mu_{X,Y}}{d\mu_X \times d\mu_Y}} \geq 0.
\end{equation}
When $(X,Y)$ has a joint density
$$
I(X;Y) = h(X) - h(X|Y),
$$
while if $Y$ is discrete and $X$ has conditional density given $Y$ then 
$$
I(X;Y) = h(X) - \sum_yp_Y(y)h(X|Y=y).
$$
The positivity of mutual information gives the standard property 
\begin{equation} \label{conditionreduceentropy}
    h(X|Y) \leq h(X).
\end{equation}

\bibliographystyle{abbrv}
\bibliography{references}

\end{document}